\documentclass[10 pt]{amsart}

\usepackage{amsmath} 
\usepackage{amsfonts}
\usepackage{amssymb}
\usepackage{amstext}
\usepackage{amsbsy}
\usepackage{amsopn}
\usepackage{amsthm}
\usepackage{amsxtra}
\usepackage{graphicx}
\usepackage{caption}
\usepackage{subcaption}
\usepackage{color}
\usepackage{hyperref}
\usepackage{enumerate}
\usepackage{amscd}

\newtheorem{theorem}{Theorem}[section]
\newtheorem{lemma}[theorem]{Lemma}
\newtheorem{corollary}[theorem]{Corollary}
\newtheorem{proposition}[theorem]{Proposition}

\newtheorem*{conjecture*}{Conjecture}
\newtheorem*{claim*}{Claim}
\newtheorem*{theorem*}{Theorem}

\theoremstyle{remark}

\theoremstyle{definition}
\newtheorem{definition}[theorem]{Definition}
\newtheorem{notation}[theorem]{Notation}
\newtheorem{example}[theorem]{Example}

\newcommand{\A}{\mathcal{A}}

\newcommand{\Z}{\mathbb{Z}}

\newcommand{\N}{\mathbb{N}}

\newcommand{\Aut}{{\rm Aut}}
\newcommand{\Hom}{{\rm Hom}}

\newcommand{\NLE}{{\Upsilon}}

\title[The automorphism group of a shift of linear growth]{The automorphism group of a shift of linear growth: beyond transitivity}
\author{Van Cyr}
\address{Bucknell University, Lewisburg, PA 17837 USA}
\email{van.cyr@bucknell.edu}
\author{Bryna Kra}
\address{Northwestern University, Evanston, IL 60208 USA}
\email{kra@math.northwestern.edu}

\subjclass[2010]{}


\thanks{The  second author was partially supported by NSF grant.}

\begin{document}

\begin{abstract}
For a finite alphabet $\mathcal{A}$ and shift $X\subseteq\mathcal{A}^{\mathbb{Z}}$ whose factor complexity function grows at most linearly, we study the algebraic properties of the automorphism group $\Aut(X)$.  For
such systems, we show that every finitely generated subgroup of $\Aut(X)$ is virtually $\Z^d$, in contrast to the behavior
when the complexity function grows more quickly. 
With additional dynamical assumptions we show more: if $X$ is transitive, then $\Aut(X)$ is virtually $\Z$; if $X$ has dense aperiodic points, then $\Aut(X)$ is virtually $\Z^d$.  We also classify all finite groups that arise as the automorphism group of a shift.
\end{abstract}

\maketitle

\section{Introduction}
Given a finite alphabet $\A$, a shift  system $(X,\sigma)$ is a closed set $X\subseteq\A^\Z$ that is invariant under the left shift $\sigma\colon 
\A^\Z\to\A^\Z$ and its automorphism group $\Aut(X)$ 
is the group of homeomorphisms of $X$ that commute with $\sigma$ (these notions are made precise in Section~\ref{sec:notation}).  For general shift systems, while $\Aut(X)$ is countable, it can be quite complicated: for the full shift~\cite{Hedlund2} or for mixing 
shifts of finite type~\cite{BLR}, $\Aut(X)$ is not finitely generated and is not amenable 
(see also~\cite{BK, KRW, FF, ward, hochman}).  
The assumption of topological mixing 
can be used to construct a rich collection of 
subgroups of the automorphism group.  For 
example, the automorphism group contains isomorphic 
copes of all finite groups, the direct sum of countably many copies of $\Z$, and the free group on two generators.
In these examples, the topological entropy is positive, and 
the complexity function $P_X(n)$,
which counts the number of nonempty cylinder sets of length $n$ taken 
over all elements $x\in X$, grows quickly.  

When the complexity function of a shift system grows slowly, the automorphism group is often much simpler and the main goal of this paper is to study the algebraic properties of $\Aut(X)$ in this setting.   
In contrast to mixing shifts, we study general shifts of low complexity, without an assumption of minimality or transitivity.  
We show that the automorphism group of any shift of low complexity is amenable, yet its behavior can still be 
be quite complicated.

As $P_X(n)$ is non-decreasing, boundedness is the slowest possible growth property that $P_X(n)$ can have.  As 
expected, this case is simple: the Morse-Hedlund Theorem~\cite{MH} implies that if there exists $n\in\N$ such that $P_X(n)\leq n$, then $X$ is comprised entirely of periodic points.  Thus $\Aut(X)$ is a finite group (and we classify all finite groups that arise in this way in Section~\ref{sec:periodic2}).  It follows that  if $(X,\sigma)$ is a shift for which $P_X(n)/n\xrightarrow{n\to\infty}0$, then $|\Aut(X)|<\infty$.

It is thus natural to study shifts for which $P_X(n) > n$ for all $n\in\N$.  
The first nontrivial growth rate that such a system can have is linear, by which we mean 
$$
0<\limsup_{n\to\infty}\frac{P_X(n)}{n}<\infty.  
$$
In previous work~\cite{CK3}, we studied the algebraic properties of $\Aut(X)$ for transitive shifts of subquadratic growth and showed that $\Aut(X)/\langle\sigma\rangle$ is a periodic group.  In particular, this holds for transitive shifts of linear growth.  Periodic groups, however, can be quite complicated: for example, a periodic group need not be finitely generated, and there are finitely generated, nonamenable periodic groups.  In this paper, we study $\Aut(X)$ for general (not necessarily transitive) shifts of linear growth.  In the transitive case, we prove a stronger result than is implied by~\cite{CK3}, 
showing that $\Aut(X)/\langle\sigma\rangle$ is finite.  
However, the main novelty of this work is that our techniques remain valid even without the assumption of transitivity.

Depending on dynamical assumptions on the system, shift systems with linear growth exhibit different behavior.  Our most general result is: 
\begin{theorem}\label{thm:main}
Suppose $(X,\sigma)$ is a shift system for which there exists $k\in\N$ such that 
$$
\limsup_{n\to\infty}P_X(n)/n<k.  
$$
Then every finitely generated subgroup of $\Aut(X)$ is virtually $\Z^d$ for some $d<k$. 
\end{theorem}

Let $[\sigma]$ denote the full group of a shift $(X,\sigma)$ (see Section~\ref{sec:full-group} for the definition). 
With the additional assumption that $(X,\sigma$) has a dense set of aperiodic points, we have: 
\begin{theorem}
\label{th:finitely-generated}
Suppose $(X,\sigma)$ is a shift system for which there exists $k\in\N$ such that 
$$
\limsup_{n\to\infty}P_X(n)/n<k.  
$$
If $X$ has a dense set of aperiodic points, then $\Aut(X)\cap[\sigma]\cong\Z^d$ for some $d<k$ and $\Aut(X)/\Aut(X)\cap[\sigma]$ is finite.  In particular, $\Aut(X)$ is virtually $\Z^d$.
\end{theorem}

For a shift $(X, \sigma)$, let $\langle\sigma\rangle$ denote the subgroup of $\Aut(X)$ generated by $\sigma$.  
With the additional assumption that $(X,\sigma)$ is topologically transitive, meaning there exists of a point whose orbit is dense in $X$, we show: 
\begin{theorem}
\label{thm:transitive}
Suppose $(X,\sigma)$ is a transitive shift system for which  
$$
0<\limsup_{n\to\infty}P_X(n)/n<\infty.  
$$
Then $\Aut(X)/\langle\sigma\rangle$ is finite.  In particular, $\Aut(X)$ is virtually $\mathbb{Z}$.
\end{theorem}

For minimal shifts, meaning shifts such that every point has dense orbit, we show (note the growth condition 
on the complexity only assumes $\liminf$ instead of $\limsup$): 
\begin{theorem}
\label{theorem:minimal}
Suppose $(X,\sigma)$ is a minimal shift for which there exists $k\in\N$ satisfying 
$$
\liminf_{n\to\infty}P_X(n)/n<k.   
$$
Then $\Aut(X)/\langle\sigma\rangle$ is finite and $|\Aut(X)/\langle\sigma\rangle|<k$.
\end{theorem}
For periodic minimal shifts, it is easy to see that $\Aut(X)\cong\Z/nZ$ where $n$ is the minimal period.  
Salo and T\"orm\"a~\cite{SaTo} asked if the automorphism group of any linearly recurrent shift is 
virtually $\mathbb{Z}$.   Linearly recurrent shifts are minimal and the factor complexity function grows at most linearly, 
and so Theorem~\ref{theorem:minimal} gives an affirmative answer to their question.

Roughly speaking, the proof of Theorem~\ref{thm:main} splits into two parts.  We start by studying shifts 
with a dense set of aperiodic points in Section~\ref{sec:aperiodic}, showing that the automorphism group is locally a group of polynomial growth, 
with the polynomial growth rate depending on the linear complexity assumption on the shift.  We sharpen this result to  understand transitive shifts of linear growth, leading to the proof of Theorem~\ref{thm:transitive} in Section~\ref{subsec:transitive}.
We then combine this with information on existence of aperiodic points, completing the proof of 
Theorem~\ref{thm:main} in Section~\ref{sec:general-linear}.
The proof of Theorem~\ref{theorem:minimal} in Section~\ref{sec:minimal} proceeds in a different manner, 
relying on a version of a lemma of Boshernitzan used to bound 
the number of ergodic probability measures on a shift with linear growth, which we use to bound 
the number of words in the language of the system that have multiple extensions.  

For some of these results, we are able to give examples showing that they are sharp.  These 
examples are included in Section~\ref{sec:examples}.

While writing up these results, we became aware of related work by Donoso, Durand, Maass, and Petit~\cite{DDMP}.  While some of the results obtained are the same, the methods are different and each method leads to new open directions.  

\section{Background and notation}\label{sec:notation}
\subsection{Shift systems}
We assume throughout that $\mathcal{A}$ is a  fixed finite set endowed with the discrete topology.  
If $x\in\mathcal{A}^{\mathbb{Z}}$, we denote the value of $x$ at $n\in\Z$ by $x(n)$.  
The metric $d(x,y):=2^{-\inf\{|n|\colon x(n)\neq y(n)\}}$ generates the product topology on $\mathcal{A}^{\mathbb{Z}}$ and endowed with this metric, $\A^\Z$ is a compact metric space; 
henceforth we assume this metric structure on $\A^{\Z}$.

The {\em left shift} $\sigma\colon\mathcal{A}^{\mathbb{Z}}\to\mathcal{A}^{\mathbb{Z}}$ is the map defined by $(\sigma x)(n):=x(n+1)$ and is a homeomorphism from $\mathcal{A}^{\mathbb{Z}}$ to itself.  If $X\subseteq\mathcal{A}^{\mathbb{Z}}$ is a closed, $\sigma$-invariant subset, then the pair $(X,\sigma)$ is called a {\em subshift of $\mathcal{A}^{\mathbb{Z}}$}, 
or just a {\em shift of $\mathcal{A}^{\mathbb{Z}}$}.  If the alphabet $\mathcal{A}$ is clear from the context, 
we refer to $(X,\sigma$) as just a {\em shift}. 

The set 
$$
\mathcal{O}(x):=\{\sigma^nx\colon n\in\N\} 
$$
is the {\em orbit of $x$} and we use $\overline{\mathcal{O}}(x)$ to denote its closure.
The shift $(X,\sigma)$ is {\em transitive} if there exists some $x\in X$ such that 
$\overline{\mathcal{O}}(x) = X$ and it is {\em minimal} if $\overline{\mathcal{O}}(x) = X$ 
for all $x\in X$.  
A point $x\in X$ is {\em periodic} if there exists some $n\in\N$ such that $\sigma^nx = x$ and otherwise it is said to be {\em aperiodic}.  

\subsection{Complexity of shifts}
For a shift $(X,\sigma)$ and $w=(a_{-m+1},\dots,a_{-1},a_0,a_1,$ $\dots,a_{m-1})\in\mathcal{A}^{2m+1}$, the {\em central cylinder set $[w]_0$ determined by $w$} is defined to be 
$$
[w]_0:=\left\{x\in X\colon x(n)=a_n\text{ for all }-m<n<m\right\}. 
$$
The collection of central cylinder sets forms a basis for the topology of $X$.  If $w=(a_0,\dots,a_{m-1})\in\mathcal{A}^m$, then the {\em one sided cylinder set $[w]_0^+$ determined by $w$} is given by 
$$
[w]_0^+:=\left\{x\in X\colon x(n)=a_n\text{ for all }0\leq n<m\right\}.
$$
For $m\in\N$,  define the set of {\em words $\mathcal{L}_m(X)$ of length $m$ in $X$} by  
$$
\mathcal{L}_m(X):=\left\{w\in\mathcal{A}^m\colon[w]_0^+\neq\emptyset\right\} 
$$
and define the {\em language $\mathcal{L}(X)$} of $X$ to be $\mathcal{L}(X):=\bigcup_{m=1}^{\infty}\mathcal{L}_m(X)$.  For $w\in\mathcal{L}(X)$, we denote the length of $w$ by $|w|$.  A word in $x\in X$ is also referred to as a {\em 
factor} of $x$.

A measure of the complexity of $X$ is the {\em (factor) complexity function} $P_X\colon X\to\N$, which counts the number of words of length $n$ in the language of $X$: 
$$
P_X(n):=|\mathcal{L}_n(X)|.  
$$
If $P_x(n)$ is the complexity function of a fixed $x\in X$, meaning 
it is the number of configurations in a block of size $n$ in $x$, then $P_X(n) \geq \sup_{x\in X}P_x(n)$, with equality holding when $X$ is a transitive shift.  

\subsection{The automorphism group of a shift}
Let $\Hom(X)$ denote the group of homeomorphisms from $X$ to itself.  If $h_1,\dots,h_n\in\Hom(X)$, then $\langle h_1,\dots,h_n\rangle$ denotes the subgroup of $\Hom(X)$ generated by $h_1,\dots,h_n$.  Thus the shift $\sigma\in\Hom(X)$ and its centralizer in $\Hom(X)$ is called the {\em automorphism group} of $(X,\sigma)$.  We denote the automorphism group of $(X,\sigma)$ by $\Aut(X)$ and endow it with the discrete topology.  

A map $\varphi\colon X\to X$ is a {\em sliding block code} if there exists $R\in\N$ such that for any $w\in\mathcal{L}_{2R+1}(X)$ and any $x,y\in[w]_0$, we have $(\varphi x)(0)=(\varphi y)(0)$.  Any number $R\in\N\cup\{0\}$ for which this property holds is called a {\em range} for $\varphi$.  The {\em minimal range} of $\varphi$ is its smallest range.

If $\varphi\colon X\to X$ is a sliding block code of range $R$, there is a natural map (which, by abuse of notation, we also denote by  $\varphi$) taking  
$\bigcup_{m=2R+1}^{\infty}\mathcal{L}_m(X)$ to $\mathcal{L}(X)$.  To define this  
extension of $\varphi$, let $m>2R$ and let $w=(a_0,\dots,a_{m-1})\in\mathcal{A}^m$.  For $0\leq i<m-2R$, choose $x_i\in[(a_i,\dots,a_{i+2R})]_0$ and define 
$$
\varphi(w):=\bigl((\varphi x_0)(0),(\varphi x_1)(0),\dots,(\varphi x_{m-2R-1})(0)\bigr).  
$$
Therefore if $w$ is a word of length at least $2R+1$, then $\varphi(w)$ is a word of length $|w|-2R$.

The elements of $\Aut(X)$ have a concrete characterization:
\begin{theorem}[Curtis-Hedlund-Lyndon Theorem~\cite{Hedlund2}]
\label{th:CHL}
If $(X,\sigma)$ is a shift, then any element of $\Aut(X)$ is a sliding block code.
\end{theorem}

For $R\in\N\cup\{0\}$,  we let $\Aut_R(X)\subseteq \Aut(X)$ denote the automorphisms of $(X,\sigma)$ for which $R$ is a (not necessarily minimal) range.  
Thus $\Aut(X)=\bigcup_{R=0}^{\infty}\Aut_R(X)$.  We observe that if $\varphi_1\in\Aut_{R_1}(X)$ and $\varphi_2\in\Aut_{R_2}(X)$, then $\varphi_1\circ\varphi_2\in\Aut_{R_1+R_2}(X)$.

In general, the automorphism group of a shift can be complicated, but 
Theorem~\ref{th:CHL} implies that $\Aut(X)$ is always countable.  

\subsection{Automorphisms and the full group}
\label{sec:full-group}

The {\em full group} $[\sigma]$ of a shift $(X,\sigma)$ is the subgroup of $\Hom(X)$ comprised of the orbit preserving homeomorphisms: 
$$
[\sigma]:=\left\{\psi\in\Hom(X):\psi(x)\in\mathcal{O}(x)\text{ for all }x\in X\right\}.
$$
Thus if $\psi\in[\sigma]$, then there is a function $k_{\psi}\colon X\to\Z$ such that $\psi(x)=\sigma^{k_{\psi}(x)}(x)$ for all $x\in X$.  

It follows from the definitions that the group $\Aut(X)\cap[\sigma]$ is the centralizer of $\sigma$ in $[\sigma]$.  
We note two basic facts  about $\Aut(X)\cap[\sigma]$ which we will need in order to study $\Aut(X)/\Aut(X)\cap[\sigma]$ in Section~\ref{lemma:aperiodic-finite}.  

\begin{lemma}
\label{lemma:normal}
If $(X,\sigma)$ is a shift, then 
$\Aut(X)\cap[\sigma]$ is normal in $\Aut(X)$.
\end{lemma}
\begin{proof}
Let $\varphi\in\Aut(X)$ and suppose $\psi\in\Aut(X)\cap[\sigma]$.  Let $k_{\varphi}\colon X\to\Z$ be a function such that $\varphi(x)=\sigma^{k_{\varphi}(x)}(x)$ for all $x\in X$.  Fix $x\in X$ and observe that since $\varphi$ and $\sigma$ commute,
$$
\varphi\circ\psi\circ\varphi^{-1}(x)=\varphi\circ\sigma^{k_{\varphi}(\varphi^{-1}(x))}\circ\varphi^{-1}(x)=\sigma^{k_{\varphi}(\varphi^{-1}(x))}(x).  
$$
As this holds for any $x\in X$, it follows that $\varphi\circ\psi\circ\varphi^{-1}\in\Aut(X)\cap[\sigma]$.  Since $\phi\in\Aut(X)$ and $\psi\in\Aut(X)\cap[\sigma]$ are arbitrary, we have 
$$\Aut(X)\cap[\sigma]=\varphi\cdot\left(\Aut(X)\cap[\sigma]\right)\cdot\varphi^{-1}$$
for all $\varphi\in\Aut(X)$.  So $\Aut(X)\cap[\sigma]$ is normal in $\Aut(X)$.
\end{proof}

\begin{lemma}\label{lemma:abelian}
If $(X,\sigma)$ is a shift, then 
$\Aut(X)\cap[\sigma]$ is abelian.
\end{lemma}
\begin{proof}
Suppose $\varphi_1, \varphi_2\in\Aut(X)\cap[\sigma]$.  For $i=1,2$, let $k_{\varphi_i}\colon X\to\Z$ be functions such that $\varphi_i(x)=\sigma^{k_{\varphi_i(x)}}(x)$ for all $x\in X$.  For any $x\in X$,  
\begin{eqnarray*}
\varphi_1\circ\varphi_2(x)&=&\varphi_1\circ\sigma^{k_{\varphi_2}(x)}(x)=\sigma^{k_{\varphi_2}(x)}\circ\varphi_1(x) \\
&=&\sigma^{k_{\varphi_2}(x)}\circ\sigma^{k_{\varphi_1}(x)}(x)=\sigma^{k_{\varphi_1}(x)}\circ\sigma^{k_{\varphi_2}(x)}(x) \\
&=&\sigma^{k_{\varphi_1}(x)}\circ\varphi_2(x)=\varphi_2\circ\sigma^{k_{\varphi_1}(x)}(x) \\
&=&\varphi_2\circ\varphi_1(x).  
\end{eqnarray*}
Therefore $\varphi_1\circ\varphi_2=\varphi_2\circ\varphi_1$.
\end{proof}

\subsection{Summary of group theoretic terminology} 
For convenience, we summarize the algebraic properties that we prove $\Aut(X)$ may have.  
We say that a group $G$ is {\em locally $P$} if every finitely generated subgroup of $G$ has property $P$.  
The group $G$ is {\em virtually $H$} if $G$ contains $H$ as a subgroup of finite index.  
The group $G$ is {\em $K$-by-$L$} if there exists a normal subgroup $H$ of $G$ which is $K$ and such that the quotient 
$G/H$ is $L$.  

\section{Shifts of linear growth with a dense set of aperiodic points}
\label{sec:transitive}

\subsection{Cassaigne's characterization of linear growth}

Linear growth can be characterized in terms of the (first) difference of the complexity function: 
\begin{theorem}[Cassaigne~\cite{C}]\label{theorem:cassaigne}
A shift $(X,\sigma)$ satisfies $p_X(n)=O(n)$ if and only if the difference function $p_X(n+1)-p_X(n)$ is bounded.
\end{theorem}

\begin{definition}
Let $w=(a_0,\dots,a_{|w|-1})\in\mathcal{L}_{|w|}(X)$.  For fixed $m\in\N$, we say that $w$ {\em extends uniquely $m$ times to the right} if there is exactly one word $\widetilde{w}=(b_0,\dots,b_{|w|+m-1})\in\mathcal{L}_{|w|+m}(X)$ such that $a_i=b_i$ for all $0\leq i<|w|$.
\end{definition}

\begin{corollary}\label{corollary:extend}
Assume $(X,\sigma)$ satisfies $p_X(n)=O(n)$. 
Then for any $m,n\in\N$, the number of words of length $n$ that do not extend uniquely $m$ times to the right is at most $B m$, where 
$B=\max_{n\in\N}\bigl(p_X(n+1)-p_X(n)\bigr)$.
\end{corollary}

Note that it follows from Cassaigne's Theorem that $B$ is finite.

\begin{proof}
For any $N\in\N$, the quantity $p_X(N+1)-p_X(N)$ is an upper bound on the number of words of length $N$ that do not extend uniquely to the right.  For any word $w$ of length $n$ which does not extend uniquely $m$ times to the right, there exists $0\leq k<m$ such that $w$ extends uniquely $k$ times to the right, but not $k+1$ times.  For fixed $k$, the number of words for which this is the case is at most the number of words of length $n+k$ that do not extend uniquely to the right.   So the number of words of length $n$ that fail to extend uniquely $m$ times to the right is at most 
\begin{equation*}
\sum_{k=1}^m\bigl(p_X(n+k)-p_X(n+k-1)\bigr)\leq Bm.\hfill\qedhere
\end{equation*}
\end{proof}

\subsection{Assuming a dense set of aperiodic points}
\label{sec:aperiodic}

We start by considering shifts with a dense set of aperiodic points.  This 
assumption holds in particular when the shift has no isolated points:  if $X$ has no 
isolated points then, for any fixed period, the set of periodic points with that period has 
empty interior.  Then the Baire Category Theorem implies that the set of all 
periodic points has empty interior.  In particular, the set of aperiodic points is dense. 
The two assumptions are equivalent if the set of aperiodic points is nonempty.  

\begin{lemma}\label{lemma:k-transitive}
Suppose $(X,\sigma)$ is a shift with a dense set of aperiodic points and there exists $k\in\N$ such that 
$$
\limsup_{n\to\infty}\frac{P_X(n)}{n}<k.  
$$
Then there exist $x_1,\dots,x_{k-1}\in X$ such that 
$$
X=\overline{\mathcal{O}}(x_1)\cup\overline{\mathcal{O}}(x_2)\cup\dots\cup\overline{\mathcal{O}}(x_{k-1}).
$$
\end{lemma}
\begin{proof}
Suppose not and let $x_1\in X$.  Since $\overline{\mathcal{O}}(x_1)\neq X$, there is a word $w_1\in\mathcal{L}(X)$ such that $[w_1]_0^+\cap\overline{\mathcal{O}}(x_1)=\emptyset$.  Choose $x_2\in X$ with $x_2\in[w_1]_0^+$.  
Let $i<k$ and suppose that we have constructed $x_1,\dots,x_i\in X$ and $w_1,\dots,w_{i-1}\in\mathcal{L}(X)$ such that $[w_{j_1}]_0\cap\overline{\mathcal{O}}(x_{j_2})=\emptyset$ whenever $j_2\leq j_1$.  Since $\overline{\mathcal{O}}(x_1)\cup\dots\cup\overline{\mathcal{O}}(x_i)\neq X$, there is a word $w_i\in\mathcal{L}(X)$ such that $[w_i]_0^+\cap\overline{\mathcal{O}}(x_1)\cup\dots\cup\overline{\mathcal{O}}(x_i)=\emptyset$.  Let $x_{i+1}\in[w_i]_0^+$ and 
we continue this construction until $i=k$.

Let $N>\max_{1\leq i<k}|w_i|$ be a fixed large integer (to be specified later).  Since $x_1$ is aperiodic, there are at least $N+1$ distinct factors of length $N$ in $\overline{\mathcal{O}}(x_1)$.  Therefore there are at least $N+1$ distinct factors of length $N$ in $X$ which do not contain the words $w_1,\dots,w_{k-1}$.  We claim that for $1\leq i<k$, 
there are at least $N-|w_i|$ distinct factors in $\overline{\mathcal{O}}(x_{i+1})$ which contain the word $w_i$ but do not contain any of the words $w_{i+1},w_{i+2},\dots,w_{k-1}$.  Assuming this claim,  then for any sufficiently large $N$ we have $p_X(N)\geq kN-\sum_{i=1}^k|w_i|$, a contradiction of the complexity assumption.

We are left with proving the claim. 
Let $1\leq i<k$ be fixed.  By construction, the word $w_i$ appears in $\overline{\mathcal{O}}(x_{i+1})$ but $[w_j]_0^+\cap\overline{\mathcal{O}}(x_{i+1})=\emptyset$ for any $j>i$.  If $w_i$ appears syndetically in $x_{i+1}$ then so long as $N$ is sufficiently large, every factor of $x_{i+1}$ of length $N$ contains the word $w_i$.  In this case, since $x_{i+1}$ is aperiodic, there are at least $N+1$ distinct factors in $\overline{\mathcal{O}}(x_{i+1})$ which contain $w_i$ but not $w_j$ for any $j>i$.  Otherwise $w_i$ does not appear syndetically in $x_{i+1}$ and so there are arbitrarily long factors in $x_{i+1}$ which do not contain $w_i$.  Since $w_i$ appears at least once in $x_{i+1}$, it follows that there are arbitrarily long words which appear in $x_{i+1}$ which contain exactly one occurrence of $w_i$ and we can assume that $w_i$ occurs as either the rightmost or leftmost subword.  Without loss, we assume that there exists a word $w$ of length $N$ which contains $w_i$ as its rightmost subword and has no other occurrences of $w_i$.  Choose $j\in\Z$ such that 
$$
w=\bigl(x_{i+1}(j),x_{i+1}(j+1),\dots,x_{i+1}(j+|w|-1)\bigr).  
$$
By construction, if $0\leq s<|w|-|w_i|$ then the word 
$$
w^{(s)}:=\bigl(x_{i+1}(j+s),x_{i+1}(j+s+1),\dots,x_{i+1}(j+s+|w|-1)\bigr) 
$$
is a word of length $N$ for which the smallest $t\in\{0,\dots,|w|-|w_i|\}$ such that 
$$
w_i=\bigl(x_{i+1}(j+t),x_{i+1}(j+t+1),\dots,x_{i+1}(j+t+|w_i|-1\bigr) 
$$
is $t=|w|-|w_i|-s$.  Therefore, the words $w^{(s)}$ are pairwise distinct and each contains $w_i$ as a subword.  By construction, they do not contain $w_j$ for any $j>i$, thus establishing the claim.
\end{proof}

\begin{proposition}\label{prop:polynomial}
Suppose that $(X,\sigma)$ is a shift with a dense set of aperiodic points and there exists $k\in\N$ such that 
$$
\limsup_{n\to\infty}\frac{P_X(n)}{n}<k.  
$$
Then $\Aut(X)$ is locally a group of polynomial growth with  polynomial growth rate is at most $k-1$.  
Moreover, if $q\in\N$ is the smallest cardinality of a set $x_1,\dots,x_q\in X$ such that $\mathcal{O}(x_1)\cup\mathcal{O}(x_2)\cup\cdots\cup\mathcal{O}(x_q)$ is dense in $X$, then the polynomial growth rate of any finitely generated subgroup of $\Aut(X)$ is at most $q$.
\end{proposition}

In Section~\ref{sec:large-poly-growth}, we give an example showing that the growth rate given in this proposition is optimal. 

\begin{proof}
By Lemma~\ref{lemma:k-transitive}, there exist $y_1,\dots,y_{k-1}\in X$ such that the union of the orbits $\mathcal{O}(y_1)\cup\mathcal{O}(y_2)\cup\cdots\cup\mathcal{O}(y_{k-1})$ is dense in $X$.  Let $x_1,\dots,x_q\in X$ be a set of minimum cardinality for which $\mathcal{O}(x_1)\cup\mathcal{O}(x_2)\cup\cdots\cup\mathcal{O}(x_q)$ is dense.  

For $1\leq i\leq q$, define the constant 
$$
C_i:=\inf\{|w|\colon\text{$w$ is a subword of $x_i$ and $[w]_0^+$ contains precisely one element}\} 
$$
and define $C_i:=0$ if no such subword exists.  Define 
\begin{equation}\label{eq:constant}
C:=\max_{1\leq i\leq q}C_i.  
\end{equation}
Fix $R\in\N$.  
For $i=1,\ldots, q$, let $\tilde{w}_i$ be a factor of $x_i$ such that 
	\begin{enumerate} 
	\item $|\tilde{w}_i|\geq3R+1$; 
	\item \label{it:two}  for all $u\in\mathcal{L}_{2R+1}(X)$, there exists $i$ such that $u$ is a factor of $\tilde{w}_i$.  
	\end{enumerate}
Note that~\eqref{it:two} is possible since $\mathcal{O}(x_1)\cup\mathcal{O}(x_2)\cup\cdots\cup\mathcal{O}(x_q)$ is dense.

Without loss of generality, we can assume that there exists $M_1\geq0$ such that $[\tilde{w}_i]_0^+$ contains precisely one element for all $i\leq M_1$ and contains at least two elements for all $i>M_1$ (otherwise reorder $x_1,\dots,x_{k-1}$).  For each $i>M_1$, either there exists $a\geq0$ such that $\tilde{w}_i$ extends uniquely to the right $a$ times but not $a+1$ times, or there exists $a\geq0$ such that $\tilde{w}_i$ extends uniquely to the left $a$ times but not $a+1$ times.  Again, reordering if necessary, we can assume that there exists $M_2\geq M_1$ such that the former occurs for all $M_1<i\leq M_2$ and the latter occurs when $i>M_2$.  
For $i=1, \ldots, q$, we define words $w_1,\dots,w_q$ as follows: 
	\begin{enumerate}
	\item For $i=1, \ldots, M_1$,  the set $[\tilde{w}_i]_0^+$ contains precisely one element.  This must be a shift of $x_i$, 
	and without loss, we can assume it is $x_i$ itself.  In this case, we define $u_i$ to be the shortest subword of $x_i$ with the property that $[u_i]_0^+$ contains precisely one element and define $w_i$ to be the (unique) extension $2R+2$ times both to the right and to the left of $u_i$.  Observe that if $\varphi, \varphi^{-1}\in\Aut_R(X)$, then $\varphi^{-1}(\varphi(w_i))=u_i$.  
Since $\varphi^{-1}$ is injective and sends every element of $[\varphi(w_i)]_0^+$ to the one point set $[u_i]_0^+$, 
it follows that $[\varphi(w_i)]_0^+$ contains precisely one element 
and the word $\varphi(w_i)$ uniquely determines the word $\varphi(\tilde{w}_i)$.  
Moreover, $|u_i|\leq C$, where $C$ is the constant in~\eqref{eq:constant}, and so $|w_i|\leq C+4R+4$.
	\item For $i=M_1+1, \dots, M_2$, there exists $a_i\geq0$ such that $\tilde{w}_i$ extends uniquely to the right $a_i$ times but not $a_i+1$ times.  Define $w_i$ to be the (unique) word of length $|\tilde{w}_i|+a_i$ which has $\tilde{w}_i$ as its leftmost factor.  By choice of the ordering, $w_i$ does not extend uniquely to its right.
	\item  For $i=M_2+1, \ldots, q$, there exists $a_i\geq0$ such that $\tilde{w}_i$ extends uniquely to the left $a_i$ times but not $a_i+1$ times.  Define $w_i$ to the be (unique) word of length $|\tilde{w}_i|+a_i$ which has $\tilde{w}_i$ as its rightmost factor.  By choice of the ordering, $w_i$ does not extend uniquely to its left.
	\end{enumerate}

For $\varphi\in \Aut_R(X)$, 
we have that $\varphi(w_i)$ determines the word $\varphi(\tilde{w}_i)$ and so 
the block code determines what $\varphi$ does to every word in $\mathcal{L}_{2R+1}(X)$.  
Thus the map $\Phi\colon\Aut_R(X)\to\mathcal{L}_{|w_1|-2R}(X)\times\mathcal{L}_{|w_2|-2R}(X)\times\cdots\times\mathcal{L}_{|w_q|-2R}(X)$ defined by 
$$
\Phi(\varphi)=\bigl(\varphi(w_1),\varphi(w_2),\dots,\varphi(w_q)\bigr)
$$
is injective.
We claim that for $1\leq i\leq q$, we have 
\begin{equation}\label{eq:words}
\left|\left\{\varphi(w_i)\colon\varphi,\varphi^{-1}\in\Aut_R(X)\right\}\right|\leq Bk(C+2)(R+1) , 
\end{equation}
where $B$ is the constant appearing in Corollary~\ref{corollary:extend} and $C$ is the constant in~\eqref{eq:constant}.  

Before proving the claim, we show how to deduce the proposition from this estimate.  It follows from~\eqref{eq:words} that $|\{\Phi(\varphi)\colon\varphi, \varphi^{-1}\in\Aut_R(X)\}|\leq(Bk(C+2))^q(R+1)^q$.  Since $\Phi$ is injective, it follows that 
we have the bound
\begin{equation}\label{eq:growth-estimate}
|\{\varphi\in\Aut_R(X)\colon\varphi^{-1}\in\Aut_R(X)\}|\leq(Bk(C+2))^q(R+1)^q.  
\end{equation}

Given $\varphi_1,\dots,\varphi_m\in\Aut(X)$, choose $R\in\N$ such that $\varphi_1,\dots,\varphi_m,\varphi_1^{-1},\dots,\varphi_m^{-1}\in\Aut_R(X)$.  Then for any $n\in\N$, any $e_1,\dots,e_n\in\{-1,1\}$, and any $f_1,\dots,f_n\in\{1,\dots,m\}$, we have 
$$
\varphi_{f_1}^{e_1}\circ\varphi_{f_2}^{e_2}\circ\cdots\circ\varphi_{f_n}^{e_n}\in\Aut_{nR}(X).    
$$
In particular, if $\mathcal{S}:=\{\varphi_1,\dots,\varphi_m,\varphi_1^{-1},\dots,\varphi_m^{-1}\}$ is 
a (symmetric) generating set for $\langle\varphi_1,\dots,\varphi_m\rangle$, then any reduced word of length $n$ (with respect to $\mathcal{S}$) is an element of $\{\varphi\in\Aut_{nR}(X)\colon\varphi^{-1}\in\Aut_{nR}(X)\}$.  By~\eqref{eq:growth-estimate}, there are at most $(Bk(C+2))^q(nR+1)^q$ such words.  Therefore $\langle\varphi_1,\dots,\varphi_m\rangle$ is a group of polynomial growth and its polynomial growth rate is at most $q$.  This holds for any finitely generated subgroup of $\Aut(X)$ (where the parameter $R$ depends on the subgroup and choice of generating set, but $B$, $C$, $k$, and $q$ depend only on the shift $(X,\sigma)$).  As $q\leq k-1$, the proposition follows.

We are left with showing that~\eqref{eq:words} holds.  
There are three cases to consider, depending on the interval in which $i$ lies. 

\begin{enumerate}
\item Suppose $1\leq i\leq M_1$.  
Then $|w_i|\leq C+4R+4$ and so $\varphi(w_i)$ is a word of length $C+2R+2$.  Therefore, there are 
$$
p_X(C+2R+2)\leq k\cdot(C+2R+2)\leq k(C+2)(R+1) 
$$
possibilities for the word $\varphi(w_i)$.

\item \label{case:two} 
Suppose $M_1<i\leq M_2$.  Then $w_i$ does not extend uniquely to its right.  If $\varphi\in\Aut(X)$ is such that $\varphi, \varphi^{-1}\in\Aut_R(X)$, then the word $\varphi(w_i)\in\mathcal{L}_{|w_i|-2R}(X)$ cannot extend uniquely $R+1$ times to its right (as otherwise this extended word would have length $2R+1$ and applying $\varphi^{-1}$ to it would show that there is only one possible extension of $w_i$ to its right).  By Corollary~\ref{corollary:extend}, there are at most $B(R+1)$ such words.  Therefore $\{\varphi(w_i)\colon\varphi,\varphi^{-1}\in\Aut_R(X)\}$ has at most $B(R+1)$ elements.

\item Suppose $i>M_2$.  Then $w_i$ does not extend uniquely to its left.  As in Case~\eqref{case:two}, if $\varphi\in\Aut(X)$ is such that $\varphi, \varphi^{-1}\in\Aut_R(X)$, then $\varphi(w_i)$ cannot extend uniquely $R+1$ times to its left.  By Corollary~\ref{corollary:extend}, there are at most $B(R+1)$ such words.  Therefore $\{\varphi(w_i)\colon\varphi,\varphi^{-1}\in\Aut_R(X)\}$ has at most $B(R+1)$ elements.  
\end{enumerate}
This establishes~\eqref{eq:words}, and thus the proposition. 
\end{proof}

\subsection{The automorphism group of a transitive shift}
\label{subsec:transitive}

\begin{lemma}\label{lemma:transitive-finite-index}
Suppose $(X,\sigma)$ is a transitive shift and there exists $k\in\N$ such that 
$$
\limsup_{n\to\infty}\frac{P_X(n)}{n}<k.  
$$
If $x_0\in X$ has a dense orbit, then the set 
$$
\left\{\varphi(x_0)\colon\varphi\in\Aut(X)\right\} 
$$
is contained in the union of finitely many distinct orbits.
\end{lemma}
\begin{proof}
Suppose not.  Let $\varphi_1,\varphi_2,\ldots\in\Aut(X)$ be such that $\varphi_i(x_0)\notin\mathcal{O}(\varphi_j(x_0))$ whenever $i\neq j$.  
For $N\in\N$, let $R(N)$ be the smallest integer such that we have $\varphi_1,\dots,\varphi_N,\varphi_1^{-1},\dots,\varphi_N^{-1}\in\Aut_{R(N)}(X)$.  For $1\leq i\leq N$, $m\in\N$, and $-n\leq j\leq n$, we have $\varphi_i^{\pm1}\circ\sigma^j\in\Aut_{R(N)+n}(X)$.  As automorphisms take aperiodic points to aperiodic points, for fixed $i$, the set 
$$
\{\varphi_i\circ\sigma^j\colon-n\leq j\leq n\} 
$$
contains $2n+1$ elements.  If $i_1\neq i_2$ and $-n\leq j_1,j_2\leq n$, then $\varphi_{i_1}\circ\sigma^{j_1}(x_0)\notin\mathcal{O}(\varphi_{i_2}\circ\sigma^{j_2}(x_0))$.  Thus the set 
$$
\{\varphi_i\circ\sigma^j\colon 1\leq i\leq N\text{ and }-n\leq j\leq n\} 
$$
contains $2Nn+N$ elements.  Therefore, 
$$
|\{\varphi\in\Aut_{R(N)+n}(X)\colon\varphi^{-1}\in\Aut_{R(N)+n}(X)\}|\geq2Nn+N.  
$$
It follows that 
$$
\limsup_{R\to\infty}\frac{|\{\varphi\in\Aut_R(X)\colon\varphi^{-1}\in\Aut_R(X)\}|}{R}\geq 2N.  
$$
Since $N\in\N$ was arbitrary, we have 
\begin{equation}\label{eq:slow-growth}
\limsup_{R\to\infty}\frac{|\{\varphi\in\Aut_R(X)\colon\varphi^{-1}\in\Aut_R(X)\}|}{R}=\infty.   
\end{equation}

On the other hand, since $(X,\sigma)$ is transitive, the parameter $q$ in the conclusion of Proposition~\ref{prop:polynomial} is $1$.  
Then by~\eqref{eq:growth-estimate}, 
we have 
$$|\{\varphi\in\Aut_R(X)\colon\varphi^{-1}\in\Aut_R(X)\}\leq Bk(C+2)(R+1), 
$$
where $B,k,C$ are as in Proposition~\ref{prop:polynomial}, which depend only on the shift $(X,\sigma)$ and not on $R$.  This estimate holds for any $R\in\N$, a contradiction of~\eqref{eq:slow-growth}.
\end{proof}


We use this to complete the proof of Theorem~\ref{thm:transitive}, characterizing 
the automorphism group of transitive shifts of linear growth:
\begin{proof}[Proof of Theorem~\ref{thm:transitive}]
Assume $(X,\sigma)$ is a transitive shift satisfying
$$\limsup_{n\to\infty}P_X(n)/n<k$$ for some $k\in\N$.  
An automorphism in a transitive shift is determined by the image of a point whose orbit is dense, and so 
Lemma~\ref{lemma:transitive-finite-index}  implies that 
the group $\Aut(X)/\Aut(X)\cap[\sigma]$ is finite (Lemma~\ref{lemma:normal} implies that $\Aut(X)\cap[\sigma]$ is normal in $\Aut(X)$).  However, the only orbit preserving automorphisms in a transitive shift are elements of $\langle\sigma\rangle$, since such an automorphism acts like a power of the shift on a point whose orbit is dense.
\end{proof}

Theorem~\ref{thm:transitive} shows that if $(X,\sigma)$ is transitive and has low enough complexity, then $\Aut(X)$ is highly constrained.  One might hope to have a converse to this theorem: if $(X,\sigma)$ is transitive and is above some ``complexity threshold'' then $\Aut(X)$ is nontrivial.  In Section~\ref{sec:no-complexity-threshold}, we give an example showing that no such converse holds.

\subsection{The automorphism group of a shift with dense aperiodic points} 
\label{sec:dense-aperiodic}

\begin{lemma}\label{lemma:aperiodic-finite}
Suppose $(X,\sigma)$ has a dense set of aperiodic points and there exists $k\in\N$ such that 
$$
\limsup_{n\to\infty}\frac{P_X(n)}{n}<k.  
$$
Let $x_1,\dots,x_q\in X$ be a set (of minimal cardinality) such that $\mathcal{O}(x_1)\cup\cdots\cup\mathcal{O}(x_q)$ is dense in $X$.  Then for each $1\leq i\leq q$, the set 
$$
\left\{\varphi(x_i)\colon\varphi\in\Aut(X)\right\} 
$$
is contained in the union of finitely many distinct orbits.
\end{lemma}
\begin{proof}
By minimality of the set $\{x_1,\dots,x_q\}$, we have 
$$
x_i\notin\bigcup_{j\neq i}\overline{\mathcal{O}}(x_j) 
$$
for any $1\leq i\leq q$.  
Therefore there exists $w_i\in\mathcal{L}(X)$ such that $[w_i]_0^+\cap\mathcal{O}(x_i)\neq\emptyset$ but $[w_i]_0^+\cap\bigcup_{j\neq i}\overline{\mathcal{O}}(x_j) =\emptyset$.  This implies that $[w_i]_0^+\subseteq\overline{\mathcal{O}}(x_i)$.  

Let $\varphi\in\Aut(X)$ and note that $\varphi$ is determined by $\varphi(x_1),\dots,\varphi(x_q)$.  
If for some $1\leq i\leq q$ we have $\mathcal{O}(\varphi(x_j))\cap[w_i]_0^+=\emptyset$ for all $j$, then $\varphi(X)\cap[w_i]_0^+=\emptyset$ and $\varphi$ is not surjective, a contradiction.  
Therefore, for each $i$ there exists $1\leq j_i\leq q$ such that $\mathcal{O}(\varphi(x_{j_i}))\cap[w_i]_0^+\neq\emptyset$.  
By construction, if $\mathcal{O}(\varphi(x_{j_i}))\cap[w_i]_0^+\neq\emptyset$, then $\varphi(x_{j_i})\in\overline{\mathcal{O}}(x_i)$ and so $\mathcal{O}(\varphi(x_{j_i}))\cap[w_k]_0^+=\emptyset$ for any $k\neq i$.  That is, the map $i\mapsto j_i$ is a permutation on the set $\{1,2,\dots,q\}$.  
Let $\pi_{\varphi}\in S_q$, where $S_q$ is the symmetric group on $q$ letters, denote this permutation.  

Let 
$$
H:=\{\pi_{\varphi}\colon\varphi\in\Aut(X)\}\subseteq S_q.  
$$
For each $h\in H$, choose $\varphi_h\in\Aut(X)$ such that $h=\pi_{\varphi_h}$.  
Then if $\varphi\in\Aut(X)$ and $h=\pi_{\varphi}$, the permutation induced by $\varphi_h^{-1}\circ\varphi$ is the identity.  It follows that $\varphi_h^{-1}\circ\varphi$ preserves each of the sets $\overline{\mathcal{O}}(x_1),\dots,\overline{\mathcal{O}}(x_q)$.  Consequently, for each $1\leq i\leq q$, the restriction of $\varphi_h^{-1}\circ\varphi$ to $\overline{\mathcal{O}}(x_i)$ is an automorphism of the (transitive) subsystem $\left(\overline{\mathcal{O}}(x_i),\sigma\right)$.  By Lemma~\ref{lemma:transitive-finite-index}, the set $\{\psi(x_i)\colon\psi\in\Aut(\overline{\mathcal{O}}(x_i))\}$ is contained in the union of finitely many distinct orbits.  Therefore, the set $\{\varphi_{\pi_{\varphi}}^{-1}\circ\varphi(x_i)\colon\varphi\in\Aut(X)\}$ is contained in the union of finitely many distinct orbits.  Since 
$$
\{\varphi(x_i)\colon\varphi\in\Aut(X)\}\subseteq\bigcup_{h\in H}\varphi_h\left(\{\varphi_{\pi_{\varphi}}^{-1}\circ\varphi(x_i)\colon\varphi\in\Aut(X)\}\right)
$$
and automorphisms take orbits to orbits, it follows that $\{\varphi(x_i)\colon\varphi\in\Aut(X)\}$ is contained in the union of finitely many distinct orbits.
\end{proof}

\begin{lemma}\label{lemma:rank}
Let $(X,\sigma)$ be a shift with a dense set of aperiodic points and assume that there exists $k\in\N$ such that 
$$
\limsup_{n\to\infty}\frac{P_X(n)}{n}<k.  
$$
Then $\Aut(X)\cap[\sigma]\cong\Z^d$ for some $d<k$.
\end{lemma}
\begin{proof}
By Lemmas~\ref{lemma:normal} and~\ref{lemma:abelian}, $\Aut(X)\cap[\sigma]$ is abelian and normal in $\Aut(X)$.  By Lemma~\ref{lemma:k-transitive}, there exist points $x_1,\dots,x_{k-1}\in X$ such that $\mathcal{O}(x_1)\cup\cdots\cup\mathcal{O}(x_{k-1})$ is dense in $X$.  If $\varphi\in\Aut(X)\cap[\sigma]$, then there exist $e_1(\varphi),\dots,e_{k-1}(\varphi)\in\Z$ such that $\varphi(x_i)=\sigma^{e_i(\varphi)}(x_i)$ for all $1\leq i\leq q$.  As an automorphism is determined by the images of $x_1,\dots,x_{k-1}$, the map $\varphi\mapsto(e_1(\varphi),\dots,e_{k-1}(\varphi))$ is an injective homomorphism from $\Aut(X)\cap[\sigma]$ to $\Z^{k-1}$.  
\end{proof}

\begin{proof}[Proof of Theorem~\ref{th:finitely-generated}]
By Lemma~\ref{lemma:k-transitive}, there exist $x_1,\dots,x_{k-1}\in X$ such that $\mathcal{O}(x_1)\cup\cdots\cup\mathcal{O}(x_{k-1})$ is dense in $X$.  If $\varphi\in\Aut(X)$, then $\varphi$ is determined by the values of $\varphi(x_1),\dots,\varphi(x_{k-1})$.  By Lemma~\ref{lemma:aperiodic-finite}, the set $\{\varphi(x_i)\colon\varphi\in\Aut(X)\}$ is contained in the union of finitely many distinct orbits in $X$.  Therefore, modulo orbit preserving automorphisms, there are only finitely many choices for  $\varphi(x_1),\dots,\varphi(x_{k-1})$.  It follows that the group $\Aut(X)/\Aut(X)\cap[\sigma]$ is finite.  By Lemma~\ref{lemma:rank}, $\Aut(X)\cong\Z^d$ for some $d<k$.
\end{proof}

\section{General shifts of linear growth}
\label{sec:general-linear}

\begin{lemma}\label{lemma:per-finite}
Suppose $(X,\sigma)$ is a shift and $w\in\mathcal{L}(X)$ is such that $[w]_0^+$ is infinite.  Then there exists aperiodic $x_w\in X$ such that $x_w\in[w]_0^+$.
\end{lemma}
\begin{proof}
Either $w$ occurs syndetically in every element of $[w]_0^+$ with a uniform bound on the gap, or there exists a sequence of elements of $[w]_0^+$ along which the gaps between occurrences of $w$ in $y_w$ grow.  

In the first case, the subsystem 
$$
\overline{\left\{\sigma^ix\colon x\in[w]_0^+\text{, }i\in\Z\right\}}
$$
is infinite and so contains an aperiodic point $x_w$.  Since $w$ occurs syndetically with the same bound in every element of $[w]_0^+$, it also occurs syndetically in any limit taken along elements of $[w]_0^+$, and in particular in $x_w$.

In the second case, there is an element of $x_w\in\overline{\mathcal{O}(y_w)}\cap[w]_0^+$ for which either $w$ occurs only finitely many times or infinitely many times with gaps tending to infinity in the semi-infinite word $\{x(n)\colon n\geq0\}$, or the same behavior occurs in the semi-infinite word $\{x(n)\colon n\leq0\}$.  In either case, $x_w$ is aperiodic.
\end{proof}

We use this to complete the proof of Theorem~\ref{thm:main}, characterizing the finitely generated subgroups of 
a shift of linear growth: 

\begin{proof}[Proof of Theorem~\ref{thm:main}]
Let  $(X,\sigma)$ be a shift and assume there exists $k\in\N$ such that 
$$
\limsup_{n\to\infty}\frac{P_X(n)}{n}<k.  
$$
Let 
$$
X_{NP}:=\overline{\left\{x\in X\colon\sigma^i(x)\neq x\text{ for all }i\neq0\right\}} 
$$
be the closure of the set of aperiodic points in $X$.  As automorphisms take aperiodic points to aperiodic points, every element of $\Aut(X)$ preserves $X_{NP}$.  Consequently, restriction to $X_{NP}$ defines 
a natural homomorphism $h\colon\Aut(X)\to\Aut(X_{NP})$. 

Let $\varphi_1,\dots,\varphi_N\in\Aut(X)$ and choose $R\in\N$ such that $\varphi_1,\dots,\varphi_N,\varphi_1^{-1},\dots,\varphi_N^{-1}\in\Aut_R(X)$.  By Lemma~\ref{lemma:k-transitive}, 
there exists a set $x_1,\dots,x_{k-1}\in X_{NP}$ such that 
$$
\mathcal{O}(x_1)\cup\cdots\cup\mathcal{O}(x_{k-1}) 
$$
is dense in $X_{NP}$.  Let $\{x_1,\dots,x_q\}\subseteq X_{NP}$ be a set of minimal cardinality with the property that 
$$
\mathcal{O}(x_1)\cup\cdots\cup\mathcal{O}(x_q) 
$$
is dense in $X_{NP}$.  Then for any $\varphi\in\langle\varphi_1,\dots,\varphi_N\rangle$, the restriction of $\varphi$ to $X_{NP}$ is determined by $\varphi(x_1),\dots,\varphi(x_q)$.  
By Lemma~\ref{lemma:aperiodic-finite}, for each $1\leq j\leq q$, the set 
$$
\{\varphi(x_j)\colon\varphi\in\langle\varphi_1,\dots,\varphi_N\rangle\} 
$$
is contained in the union of finitely many distinct orbits.  Therefore there exists a finite collection of automorphisms $\psi_1,\dots,\psi_M\in\langle\varphi_1,\dots,\varphi_N\rangle$ such that for any $\varphi\in\langle\varphi_1,\dots,\varphi_N\rangle$, there exists $1\leq t(\varphi)\leq M$ such that for all $1\leq j\leq q$, we have 
$$
\varphi(x_j)\in\mathcal{O}(\psi_{t(\varphi)}(x_j)).  
$$
Thus the restriction of $\psi_{t(\varphi)}^{-1}\circ\varphi$ to $X_{NP}$ is orbit preserving.  Let 
$$
K:=\{\varphi\in\langle\varphi_1,\dots,\varphi_N\rangle\colon\text{the restriction of $\varphi$ to $X_{NP}$ is orbit preserving}\}.  
$$
Clearly $K$ is a subgroup of $\langle\varphi_1,\dots,\varphi_N\rangle$.  

For each  $1\leq i\leq N$, we have that $\varphi_i$ is a block code of range $R$.  Let 
$$
\mathcal{W}_R:=\left\{w\in\mathcal{L}_{2R+1}(X)\colon[w]_0^+\cap X_{NP}=\emptyset\right\}.  
$$
Then by Lemma~\ref{lemma:per-finite}, the set 
$$
Y:=\bigcup_{w\in\mathcal{W}_R}[w]_0^+ 
$$
is finite.  Since every element of $Y$ is periodic and automorphisms preserve the minimal period of periodic points, the ($\langle\varphi_1,\dots,\varphi_N\rangle$-invariant) set 
$$
Z:=\left\{\varphi_{i_1}^{e_1}\circ\cdots\circ\varphi_{i_S}^{e_S}(y)\colon i_1,\dots,i_S\in\{1,\dots,N\}\text{, }e_1,\dots,e_S\in\{-1,1\},S\in\N\text{, }y\in Y\right\} 
$$
is finite.  For any $1\leq i\leq N$, the restriction of $\varphi_i$ to $X_{NP}$ uniquely determines the restriction of $\varphi_i$ to $X\setminus Z$ (since all words of length $2R+1$ that occur in elements of $X\setminus Z$ also occur in $X_{NP}$).  Since $\varphi_1,\dots,\varphi_N$ are automorphisms that preserve $Z$, they take elements of $X\setminus Z$ to elements of $X\setminus Z$.
Thus for any $\varphi\in\langle\varphi_1,\dots,\varphi_N\rangle$, the restriction of $\varphi$ to $X_{NP}$ uniquely determines the restriction of $\varphi$ to $X\setminus Z$.  In particular, this holds for all $\varphi\in K$.  So there exists a finite collection of automorphisms $\alpha_1,\dots,\alpha_T\in K$ such that for all $\varphi\in K$, 
there is an integer $1\leq s(\varphi)\leq T$ such that $\alpha_{s(\varphi)}^{-1}\circ\varphi$ acts trivially on $Z$.

With the functions $t(\varphi)$ and $s(\varphi)$ defined as above, we have that for any $\varphi\in\langle\varphi_1,\dots,\varphi_N\rangle$, the automorphism 
$$
\alpha^{-1}_{s\left(\psi^{-1}_{t(\varphi)}\circ\varphi\right)}\circ\psi^{-1}_{t(\varphi)}\circ\varphi 
$$
acts trivially on $Z$ and its restriction to $X_{NP}$ is orbit preserving.  Define $H\subseteq\langle\varphi_1,\dots,\varphi_N\rangle$ to be the subgroup of elements $\varphi\in\langle\varphi_1,\dots,\varphi_N\rangle$ 
such that $\varphi$ acts trivially on $Z$ and the restriction of $\varphi$ to $X_{NP}$ is orbit preserving. 
Every element of $H$ is uniquely determined by its restriction to $X_{NP}$, and so $H$ is isomorphic to a subgroup of $\Aut(X_{NP})\cap[\sigma]$.  By Lemma~\ref{lemma:rank}, this subgroup is isomorphic to $\Z^d$ for some $d<k$.  On the other hand, for any $\varphi\in\langle\varphi_1,\dots,\varphi_N\rangle$, there exist $1\leq t\leq M$ and $1\leq s\leq T$ such that $\alpha_s^{-1}\circ\psi_t^{-1}\circ\varphi\in H$.  Therefore $H$ has finite index in $\Aut(X)$.  

Finally, if $\varphi\in H$, then there is a function $k\colon X_{NP}\to\Z$ such that for all $x\in X_{NP}$ we have $\varphi(x)=\sigma^{k(x)}(x)$.  Thus if $\psi\in\langle\varphi_1,\dots,\varphi_N\rangle$ and $x\in X_{NP}$, we have 
$$
\psi\circ\varphi\circ\psi^{-1}(x)=\psi\circ\sigma^{k(\psi^{-1}(x))}\circ\psi^{-1}(x)=\sigma^{k(\psi^{-1}(x))}(x) 
$$
and if $x\in Z$, we have 
$$
\psi\circ\varphi\circ\psi^{-1}(z)=z.  
$$
Therefore $\psi\circ\varphi\circ\psi^{-1}\in H$ and so $H$ is a normal subgroup of $\Aut(X)$.  

It follows that $\Aut(X)$ is virtually $\Z^d$ for some $d<k$.
\end{proof}

\section{Minimal shifts of linear growth}
\label{sec:minimal}
For minimal shifts, we need more information on the words that are uniquely extendable: 

\begin{definition}
For $x\in X$, define 
$$
x_R:=\{y\in X\colon y(i)=x(i)\text{ for all }i\geq0\}.  
$$
For $x,y\in X$, we write $x\sim_Ry$ if $x_R=y_R$ and define $X_R:=X/\!\sim_R$ to be $X$ modulo this relation.  
\end{definition}
It is easy to check that $\sim_R$ is an equivalence relation on $X$ and 
so $X_R$ is well defined.  
We view $(X_R,\sigma)$ as a one sided shift.  If $\varphi\in\Aut(X)$, then $\varphi$ is a block code (say of range $N$) and so determines an endomorphism on $(X_R,\sigma)$ as follows: if $y\in x_R$ and $N\in\N$ is the minimal range of $\varphi$, then $\varphi(x_R):=\left(\sigma^N\circ\varphi(y)\right)_R$.  It is easy to check that $\varphi(x_R)$ is well defined.

\begin{definition}
For $x\in X$, we say that $x_R$ is {\em uniquely left extendable} if it has a unique preimage in $X_R$ and {\em nonuniquely left extendable} otherwise.

If $w\in\mathcal{L}_n(X)$ is a word of length $n$ in the language of $X$, we 
say that $w$ is {\em uniquely left extendable} if there is a unique $\hat{w}\in\mathcal{L}_{n+1}(X)$ that ends with $w$.
\end{definition}

Boshernitzan~\cite{Bos} showed that if $(X,\sigma)$ is minimal and there exists $k\in\N$ such that 
$$
\liminf_{n\to\infty}P_X(n)-kn=-\infty, 
$$
then the number of ergodic probability measures on $(X,\sigma)$ is finite.   In his proof, he makes use of a counting 
lemma and we use an infinite version of this lemma to study minimal shifts of linear growth: 

\begin{lemma}[Infinite version of Boshernitzan's Lemma]
\label{lemma-boshernitzan}
Let $(X,\sigma)$ be a shift for which there exists $k\in\N$ such that 
\begin{equation}\label{eq1}
\liminf_{n\to\infty}P_X(n)-kn=-\infty.  
\end{equation}
Then there are at most $k-1$ distinct elements of $(X_R,\sigma)$ which are nonuniquely left extendable.
\end{lemma}
\begin{proof}
We first claim that for infinitely many $n$, the number of words of length $n$ that are nonuniquely left extendable is at most $k-1$.  If not, let $L_n$ be the number of words of length $n$ that do not extend uniquely to their left.  Then by assumption there exists $N\in\N$ such that for all $n\geq N$ we have $L_n\geq k$.  However, 
$$
P_X(n+1)\geq P_X(n)+L_n, 
$$
and so $P_X(n)\geq P_X(N)+k\cdot(n-N)$ for all $n\geq N$.  This contradicts~\eqref{eq1}, and the claim follows.

We use this to show that there are at most $k-1$ elements in $X_R$ which are nonuniquely left extendable.  
If not, there exist distinct elements $x_1,\dots,x_k\in X_R$ which are all nonuniquely left extendable.  Choose $M\in\N$ such that for any $1\leq i<j\leq k$, there exists $0\leq m<M$ such that $x_i(m)\neq x_j(m)$.  By the first claim, there exists $n>M$ such that there are at most $k$ words of length $n$ that are nonuniquely left extendable.  For all $1\leq i\leq k$, the word 
$$
\left(x_i(0),x_i(1),\dots,x_i(n-2),x_i(n-1)\right) 
$$
is a word of length $n$ that is nonuniquely left extendable and these words are pairwise distinct since $n>M$, 
leading to a contradiction.  Thus the number of elements of $(X_R,\sigma)$ that are nonuniquely left extendable is at most $k-1$.
\end{proof}

\begin{notation}\label{NLE-def}
We write $\NLE_0\subseteq X_R$ for the collection of nonuniquely left extendable points in $X_R$.  For $m\in\N$, we write $\NLE_m:=\sigma^m(\NLE_0)$ for the collection of elements of $X_R$ whose preimage under $m$ iterates of $\sigma$  contains more than one point.
\end{notation}

\begin{lemma}\label{lemma-extension}
If $y\in X_R\setminus\bigcup_{m=0}^{\infty}\NLE_m$, then there is a unique $z\in X$ for which $y=z_R$.
\end{lemma}
\begin{proof}
If not, there exist distinct 
$z_1, z_2\in X$ and $y=(z_1)_R=(z_2)_R$.  Thus there exists $i\in\N$ such that $z_1(-i)\neq z_2(-i)$.  Set $i_0$ 
to be the minimal such $i$.  Then $\sigma^{-i_0+1}y=(\sigma^{-i_0+1}z_1)_R=(\sigma^{-i_0+1}z_2)_R$, \ but $(\sigma^{-i_0}z_1)_R\neq(\sigma^{-i_0}z_2)_R$.  Thus $\sigma^{-i_0+1}y\in \NLE_0$, which implies that $y\in \NLE_{-i_0+1}$, a contradiction.
\end{proof}

\begin{lemma}\label{lemma-finite-extension}
If $(X,\sigma)$ is a shift, $\varphi\in\Aut(X)$, and $y\in \NLE_0$, then there exists $m\geq 0$ such that $\varphi(y)\in \NLE_m$.
\end{lemma}
\begin{proof}
It not, then $\varphi(y)\in X_R\setminus\bigcup_{m=0}^{\infty}\NLE_m$ and so Lemma~\ref{lemma-extension} 
implies that there is a unique $z\in X$ such that $\varphi(y)=z_R$.  Since $\varphi$ is an automorphism, it follows that $\varphi^{-1}(z)$ is the only solution to the equation $y=x_R$, a contradiction of $y\in \NLE_0$.
\end{proof}

We use this to complete the characterization of the automorphism group for minimal aperiodic shifts 
with linear growth: 

\begin{proof}[Proof of Theorem~\ref{theorem:minimal}]
Assume $(X,\sigma)$ is an aperiodic minimal shift such that there exists $k\in\N$ with 
$\liminf_{n\to\infty}P_X(n)/n<k$.  

Fix $y\in \NLE_0$ and let $\varphi\in\Aut(X)$.  By Lemma~\ref{lemma-finite-extension}, there exists $m\in\N$ such that $\varphi(y)\in \NLE_m$.  Let $m_{\varphi}\geq 0$ be the smallest non-negative integer for which $\varphi(y)\in \NLE_m$.  Then there exists $z_{\varphi}\in \NLE_0$ such that $\sigma^{m_{\varphi}}(z_{\varphi})=\varphi(y)$.

Now suppose $\varphi_1, \varphi_2\in\Aut(X)$ and $z_{\varphi_1}=z_{\varphi_2}$.  We claim that $\varphi_1$ and $\varphi_2$ project to the same element in $\Aut(X)/\langle\sigma\rangle$.  Without loss, suppose $m_{\varphi_1}\leq m_{\varphi_2}$.  Then 
$$
\varphi_2(y)=\sigma^{m_{\varphi_2}}(z_{\varphi_2})=\sigma^{(m_{\varphi_2}-m_{\varphi_1})}\circ\sigma^{m_{\varphi_1}}(z_{\varphi_1})=\sigma^{(m_{\varphi_2}-m_{\varphi_1})}\circ\varphi_1(y).
$$
By minimality, every word of every length occurs syndetically in every element of $(X,\sigma)$.  It follows that
all words occur syndetically in every element of $(X_R,\sigma)$, and in particular, all words occur syndetically in $y$.  
Both $\varphi_2$ and $\sigma^{(m_{\varphi_2}-m_{\varphi_1})}\circ\varphi_1$ are sliding block codes.  Since $\varphi_2(y)=\sigma^{(m_{\varphi_2}-m_{\varphi_1})}\circ\varphi_1(y)$, it follows that $\varphi_2$ and $\sigma^{(m_{\varphi_2}-m_{\varphi_1})}\circ\varphi_1$ have the same image on every word, meaning that they define the same block code.  In other words, $\varphi_1$ and $\varphi_2$ project to the same element in $\Aut(X)/\langle\sigma\rangle$, proving the claim.

Since $|\NLE_0|\leq k-1$, Lemma~\ref{lemma-boshernitzan} implies 
that there can be at most $k-1$ distinct elements of $(X_R,\sigma)$ that arise as $z_{\varphi}$ for $\varphi\in\Aut(X)$.  Therefore, there are at most $k-1$ distinct elements of $\Aut(X)/\langle\sigma\rangle$. $\hfill\square$
\end{proof}

This can be used to characterize the automorphism groups for particular systems.  
We note the simplest case of a Sturmian shift  for later use (see~\cite[Example 4.1]{Olli}):
\begin{corollary}
\label{cor:olli}
If $(X,\sigma)$ is a Sturmian shift, then $\Aut(X)=\langle\sigma\rangle$.
\end{corollary}
\begin{proof}
For a Sturmian shift, $(X,\sigma)$ is minimal, aperiodic, and $P_X(n)=n+1$ for all $n\in\N$.  Applying Theorem~\ref{theorem:minimal} with $k=2$, we have that $\left|\Aut(X)/\langle\sigma\rangle\right|=1$.
\end{proof}

More generally: 
\begin{corollary}
If $(X,\sigma)$ is aperiodic, minimal and there exists $k\in\N$ such that 
$$
\liminf_{n\to\infty}P_X(n)-kn=-\infty, 
$$
then $\Aut(X)$ is the semi-direct product of a finite group and $\Z$.
\end{corollary}
\begin{proof}
By Theorem~\ref{theorem:minimal}, $\Aut(X)/\langle\sigma\rangle$ is finite.  Since $\langle\sigma\rangle$ 
has infinite order and is contained in the center of $\Aut(X)$, it follows from the classification of virtually cyclic groups (see~\cite{SW}) that $\Aut(X)$ is the semi-direct product of 
a finite group and $\Z$. 
\end{proof}

\section{Examples}
\label{sec:examples}

\subsection{Automorphism group with large polynomial growth}
\label{sec:large-poly-growth}
Proposition~\ref{prop:polynomial} shows that if $(X,\sigma)$ is a shift satisfying 
$$
\limsup_{n\to\infty}\frac{P_X(n)}{n}<k 
$$
then $\Aut(X)$ is locally a group of polynomial growth, with polynomial growth rate at most $k-1$.  The following Proposition shows that this estimate of the polynomial growth rate of $\Aut(X)$ is optimal.
\begin{proposition}
Let $k\in\N$ be fixed and let $\A=\{0,1\}\times\{1,\dots,k\}$.  There is a shift $X\subseteq\A^{\Z}$ with a dense set of aperiodic points such that $P_X(n)=kn+k$ and $\Aut(X)\cong\Z^k$.
\end{proposition}
\begin{proof}
Recall that a Sturmian shift is an aperiodic, minimal shift of $\{0,1\}^{\Z}$ whose complexity function satisfies $P_X(n)=n+1$ for all $n$.  There are uncountably many Sturmian shifts and any particular Sturmian shift only factors onto countably many other Sturmian shifts (since the factor map must be a sliding block code, of which there are only countably many).  Therefore there exist $k$ Sturmian shifts $X_1, X_2,\dots,X_k$ such that there exists a sliding block code taking $X_i$ to $X_j$ if and only if $i=j$.  We identify $X_i$ with in a natural way with a shift of $\A^{\Z}$ by writing the elements of $X_i$ with the letters $(0,i)$ and $(1,i)$ and will abuse notation by also referring to this shift as $X_i$.  Let $X:=X_1\cup\cdots\cup X_k$ (which is clearly shift invariant, and is closed because the minimum distance between a point in $X_i$ and $X_j$ is $1$ whenever $i\neq j$).  

Let $\varphi\in\Aut(X)$.  As $\varphi$ is given by a sliding block code, $\varphi$ must preserve the sets $X_1,\dots,X_k$.  Therefore $\Aut(X)\cong\Aut(X_1)\times\cdots\times\Aut(X_k)$.  By Corollary~\ref{cor:olli} we have $\Aut(X_i)=\langle\sigma\rangle\cong\Z$ for $i=1,\dots,k$.  So $\Aut(X)\cong\Z^k$.
\end{proof}

\subsection{Quickly growing transitive shifts with trivial automorphism group}
\label{sec:no-complexity-threshold}
Next we describe a general process which takes a minimal shift of arbitrary growth rate and produces a transitive shift with essentially the same growth, but whose automorphism group consists only of powers of the shift.  This shows that there is no ``complexity threshold'' above which the automorphism group of a transitive shift must be nontrivial.

\begin{lemma}\label{lemma:dense-orbit}
If $(X,\sigma)$ is a transitive shift with precisely one dense orbit, then $\Aut(X)=\langle\sigma\rangle$.
\end{lemma}
\begin{proof}
Suppose that there exists $x_0\in X$ such that 
$$
\{y\in X\colon y\text{ has a dense orbit}\}=\mathcal{O}(x_0).  
$$
If $\varphi\in\Aut(X)$, then $\varphi(x_0)$ has a dense orbit and so there exists $k\geq 0$ such that $\varphi(x_0)=\sigma^k(x_0)$.  It follows that $\varphi$ and $\sigma^k$ agree on the (dense) orbit of $x_0$.  Since both functions are continuous, they agree everywhere.
\end{proof}

\begin{example}

Let $\mathcal{A}=\{0,1,2,\dots,n-1\}$ and let $X\subseteq\mathcal{A}^{\mathbb{Z}}$ be a minimal shift.  Let 
$\tilde{\mathcal{A}}=\mathcal{A}\cup\{n\}$, where we add the symbol $n$ to the alphabet and $n\notin\A$.  Fix $x_0\in X$ and define $\tilde{x}_0\in\tilde{\mathcal{A}}^{\mathbb{Z}}$ by: 
$$
\tilde{x}_0(i)=\begin{cases}
x_0(i) & \text{ if }i\neq 0; \\
n & \text{ if }i=0.
\end{cases}
$$
Let $\tilde{X}\subseteq\tilde{A}^{\mathbb{Z}}$ be the orbit closure of $\tilde{x}_0$.  Then $\tilde{X}=X\cup\mathcal{O}(\tilde{x}_0)$, $(\tilde{X},\sigma)$ is transitive, $p_{\tilde{X}}(n)=p_X(n)+n$ for all $n\in\N$, and $\tilde{X}$ has precisely one dense orbit.  By  Lemma~\ref{lemma:dense-orbit},  $\Aut(\tilde{X})=\langle\sigma\rangle$.  
\end{example}

\subsection{$\Aut(X)$ and $\Aut(X)/\Aut(X)\cap[\sigma]$ are not always finitely generated}
\label{sec:not-global}

Theorem~\ref{thm:main} shows that every finitely generated subgroup of $\Aut(X)$ is virtually $\Z^d$.  When $X$ has a dense set of aperiodic points, Theorem~\ref{th:finitely-generated} shows that $\Aut(X)/\Aut(X)\cap[\sigma]$ is finite.  In this section we show that the result of Theorem~\ref{th:finitely-generated} cannot be extended to the general case, and the words ``every finitely generated subgroup'' cannot be removed from the statement of Theorem~\ref{thm:main}.  We begin with an example to set up our construction.

\begin{example}
Let $\mathcal{A}=\{0,1\}$ and for $n\in\N$, let $x_n\in\A^{\mathbb{Z}}$ be the periodic point 
	$$
	x_n(i)=\begin{cases}
	1 & \text{ if }i\equiv0\text{ (mod $2^n$)}; \\
	0 & \text{ otherwise}.
	\end{cases}
	$$
Let $X$ be the closure of the set $\{x_n\colon n\in\N\}$ under $\sigma$.  If we define 
	$$
	x_{\infty}(i)=\begin{cases}
	1 & \text{ if }i=0; \\
	0 & \text{ otherwise}, 
	\end{cases}
	$$
and ${\bf 0}$ to be the $\A$-coloring of all zeros, then we have 
$$
X=\{{\bf 0}\}\cup\mathcal{O}(x_{\infty})\cup\bigcup_{n=1}^{\infty}\mathcal{O}(x_n).  
$$
Suppose $R\in\N$ is fixed and $\varphi\in\Aut_R(X)$.  Since $\varphi$ preserves the period of periodic points, $\varphi({\bf 0})={\bf 0}$.  In particular, the block code $\varphi$ takes the block consisting of all zeros to $0$.  It follows that there exists $k\in[-R,R]$ such that $\varphi(x_{\infty})=\sigma^k(x_{\infty})$.  For any $m>2R+1$, the blocks of length $2R+1$ occurring in $x_m$ are identical to those appearing in $x_{\infty}$ and so $\varphi(x_m)=\sigma^k(x_m)$ for all such $m$.

Now let $\varphi_1,\dots,\varphi_n\in\Aut(X)$ and find $R\in\N$ such that $\varphi_1,\dots,\varphi_n,\varphi_1^{-1},\dots,\varphi_n^{-1}\in\Aut_R(X)$.  For $1\leq i\leq n$, let $k_i\in[-R,R]$ be such that for all $m>2R+1$ we have $\varphi_i(x_m)=\sigma^{k_i}(x_m)$.  Then $N\in\N$, any $e_1,\dots,e_N\in\{1,\dots,n\}$, any $\epsilon_1,\dots,\epsilon_N\in\{-1,1\}$, and any $m>2R+1$, we have 
$$
\left(\varphi_{e_1}^{\epsilon_1}\circ\varphi_{e_2}^{\epsilon_2}\circ\cdots\circ\varphi_{e_N}^{\epsilon_N}\right)(x_m)=\sigma^{(\epsilon_1\cdot k_{e_1}+\epsilon_2\cdot k_{e_2}+\cdots+\epsilon_N\cdot k_{e_N})}(x_m).  
$$
Then if $\varphi\in\Aut(X)$ is the automorphism that acts like $\sigma$ on $\mathcal{O}(x_{R+1})$ and acts trivially on $X\setminus\mathcal{O}(x_{R+1})$ (this map is continuous because $x_{R+1}$ is isolated), then $\varphi\notin\langle\varphi_1,\dots,\varphi_N\rangle$.  Therefore $\langle\varphi_1,\dots,\varphi_N\rangle\neq\Aut(X)$.  Since $\varphi_1,\dots,\varphi_n\in\Aut(X)$ were general, it follows that $\Aut(X)$ is not finitely generated.

On the other hand 
$$
P_X(n)=n+2^{\lfloor\log_2(n)\rfloor+1}-1<3n 
$$
for all $n$, so $P_X(n)$ grows linearly.  We also remark that $\Aut(X)=\Aut(X)\cap[\sigma]$ for this shift.
\end{example}

\begin{proposition}\label{ex:infinitely-generated}
There exists a shift $(X,\sigma)$ of linear growth that has a dense set of periodic points and is such that none of the groups $\Aut(X)$, $\Aut(X)\cap[\sigma]$, and $\Aut(X)/\Aut(X)\cap[\sigma]$ are finitely generated.
\end{proposition}
\begin{proof}
Let $X_1$ be the shift of $\{0,1\}^{\Z}$ constructed in the previous example.  Let $X_2$ be the same shift, constructed over the alphabet $\{2,3\}$ (by identifying $0$ with $2$ and $1$ with $3$).  Let $X=X_1\cup X_2$ and observe that $d(X_1,X_2)=1$.  Since $\Aut(X_i)\cap[\sigma]=\Aut(X_i)$ for $i=1,2$, we have $\Aut(X)\cap[\sigma]\cong\Aut(X_1)\times\Aut(X_2)$.  Therefore $\Aut(X)\cap[\sigma]$ is not finitely generated.  On the other hand, 
$$
P_X(n)=P_{X_1}(n)+P_{X_2}(n)=2\cdot P_{X_1}(n)<6n  
$$
so $X$ is a shift of linear growth (and has a dense set of periodic points).  

We claim that $\Aut(x)/\Aut(X)\cap[\sigma]$ is not finitely generated.  Define $\delta\in\Aut(X)$ to be the range $0$ involution that exchanges $0$ with $2$ and $1$ with $3$.  For each $m\in\N$ let $\delta_m\in\Aut(X)$ be the range $0$ involution which exchanges the (unique) orbit of period $2^m$ in $X_1$ with the (unique) orbit of period $2^m$ in $X_2$ by exchanging $0$ with $2$ and $1$ with $3$ in these orbits only (and fixing the remainder of $X$.  For $i\in\N$ let $\tilde{\delta}_i$ be the projection of $\delta_i$ to $\Aut(X)/\Aut(X)\cap[\sigma]$ and let $\tilde{\delta}$ be the projection of $\delta$.  These involutions commute pairwise and the set $\{\tilde{\delta}_i\colon i\in\N\}\cup\{\tilde{\delta}\}$ is clearly independent.  

Now let ${\bf x}\in X_1$ be the point ${\bf x}(i)=1$ if and only if $i=0$, and let ${\bf y}\in X_2$ be the point ${\bf y}(i)=4$ if and only if $i=0$.  Let $\varphi\in\Aut(X)$ be fixed and observe that either $\varphi({\bf x})\in\mathcal{O}({\bf x})$ or $\varphi({\bf x})\in\mathcal{O}({\bf y})$.  In the former case define $\epsilon:=0$ and in the latter case define $\epsilon:=1$, so that $\varphi\circ\delta^{\epsilon}$ preserves the orbit of ${\bf x}$ (hence also the orbit of ${\bf y}$).  As $\varphi\circ\delta^{\epsilon}$ is given by a block code which carries the block of all $0$'s to $0$, there are at most finitely many $m$ such that $\varphi\circ\delta^{\epsilon}$ does not preserve the orbit of the (unique) periodic orbit of period $2^m$ in $X_1$.  Let $m_1<\cdots<m_n$ be the set of $m$ for which it does not preserve the orbit.  Then 
$$
\varphi\circ\delta^{\epsilon}\circ\delta_{m_1}\circ\cdots\circ\delta_{m_n}\in\Aut(X)\cap[\sigma].  
$$
Therefore $\Aut(X)/\Aut(X)\cap[\sigma]$ is the group generated by $\tilde{\delta}, \tilde{\delta}_1, \tilde{\delta}_2, \tilde{\delta}_3,\dots$  This group is isomorphic to $\prod_{i=1}^{\infty}\Z_2$ so $\Aut(X)/\Aut(X)\cap[\sigma]$ is not finitely generated.  Finally, as $\Aut(X)$ factors onto a group that it not finitely generated, it is not finitely generated either.
\end{proof}

\section{Automorphisms of periodic shifts}
\label{sec:periodic2}

We characterize which finite groups arise as automorphism groups of shifts.
\begin{definition}
For $n>1$ and $m\in\N$, let 
$$
\Z_n^m:=\underbrace{\mathbb{Z}_n\times\mathbb{Z}_n\times\cdots\times\mathbb{Z}_n}_{m\text{ times}}  
$$
where $\mathbb{Z}_n$ denotes $\Z/n\Z$.  Let $S_m$ denote the symmetric group on $m$ letters and define a homomorphism $\psi\colon S_m\to\Aut(\Z_n^m)$ by 
$$
\left(\psi(\pi)\right)(i_1,\dots,i_m):=(i_{\pi(1)},\dots,i_{\pi(m)}).  
$$
Then the {\em generalized symmetric group} is defined as in~\cite{Osima} to be  
$$
S(n,m):=\Z_n^m\rtimes_{\psi} S_m.  
$$
Equivalently, $S(n,m)$ is the wreath product $\Z_n\wr S_m$.
\end{definition}

\begin{theorem}
Suppose $G$ is a finite group.  There exists a shift $(X,\sigma)$ for which $\Aut(X)\cong G$ if and only if there exist $s\in\N$, $n_1<n_2<\cdots<n_s$, and $m_1, m_2, \dots,m_s\in\N$ such that 
$$
G\cong S(n_1,m_1)\times S(n_2,m_2)\times\cdots\times S(n_s,m_s).
$$
\end{theorem}
\begin{proof}
Suppose $(X,\sigma)$ is a shift for which $\Aut(X)$ is finite.  Since $\sigma\in\Aut(X)$, there exists $k\in\N$ such that $\sigma^k(x)=x$ for all $x\in X$.  
That is, $X$ is comprised entirely of periodic points such that the minimal period of each point 
is a divisor of $k$.  Since a shift can have only finitely many such points, $X$ is finite.  Let $x_1,\dots,x_N\in X$ be representatives of the orbits in $X$, meaning that $\mathcal{O}(x_i)\cap\mathcal{O}(x_j)=\emptyset$ whenever $i\neq j$ and for all $x\in X$ there exist $i,k\in\N$ such that $x=\sigma^k(x_i)$.  For $i=1, \ldots, N$, let $p_i$ be the minimal period of $x_i$ and, without loss, assume that $p_1\leq p_2\leq\cdots\leq p_N$.  Define $n_1:=p_1$ and inductively define $n_2,n_3,\dots,n_s$ by 
$$
n_{i+1}:=\min\{p_j\colon p_j>n_i\}, 
$$
where $s$ is the number of steps before the construction terminates.  Define 
$$
m_i:=\vert\{j\colon p_j=n_i\}\vert.  
$$
Let $\varphi\in\Aut(X)$.  Then for $1\leq i\leq s$, $\varphi$ induces a permutation on the set of periodic points of minimal period $n_i$.  
More precisely, for fixed $1\leq i\leq s$, we can define $\pi_i^{\varphi}\in S_{m_i}$ to be the permutation that sends $j\in\{1,2,\dots,m_i\}$ to the unique integer $k\in\{1,2,\dots,m_i\}$ such that $\varphi(x_{m_1+\cdots+m_{i-1}+j})\in\mathcal{O}(x_{m_1+\cdots+m_{i-1}+k})$.  For $1\leq j\leq m_i$, choose $k_j^i\in\Z_{n_i}$ such that 
$$
\varphi(x_{m_1+\cdots+m_{i-1}+j})=\sigma^{k_j^i}(x_{m_1+\cdots+m_{i-1}+\pi_i^{\varphi}(j)}).  
$$
Then the map 
$$
\varphi\mapsto(k_1^1,k_2^1,\dots,k_{m_1}^1,\pi_1^{\varphi},k_1^2,\dots,k_{m_2}^2,\pi_2^{\varphi},\dots,k_1^s,\dots,k_{m_s}^s,\pi_s^{\varphi}) 
$$
is a homomorphism from $\Aut(X)$ to $S(n_1,m_1)\times\cdots\times S(n_s,m_s)$.  The kernel of this map is trivial.  
To check that it is surjective,  if $\pi_1,\dots,\pi_s$ are permutations ($\pi_i\in S_{m_i}$ for all $i$), then  define 
$$
\varphi_{\pi_1,\dots,\pi_s}(x_{m_1+\cdots+m_{i-1}+j})=x_{m_1+\cdots+m_{i-1}+\pi_i(j)} 
$$
and extend this to an automorphism of $(X,\sigma)$.  Similarly, for $1\leq i\leq N$, define 
$$
\varphi_i(\sigma^k(x_j)):=\sigma^{k+\delta_{i,j}}(x_j), 
$$
where $\delta_{i,j}$ is the Kronecker delta.  Note that each of these maps is given by 
a block code, where the range is the smallest $R$ such that $\sigma^R(x)=x$ for all $x\in X$.  
Taken together, this shows that the map $\phi$ is surjective and thus is an isomorphism.

Conversely, suppose that $n_1<\cdots<n_s$ are given and $m_1,\dots,m_s\in\N$.  For $1\leq i\leq s$ and $1\leq j\leq m_i$, define 
$$
x_{i,j}(k):=\begin{cases}
j & \text{ if }k\equiv0\text{ (mod $n_i$)}; \\
0 & \text{ otherwise}.  
\end{cases}
$$
Let 
$$
X^{\prime}=\bigcup_{i=1}^s\bigcup_{j=1}^{n_i}x_{i,j} 
$$
and let $X$ be the closure of $X^{\prime}$ under $\sigma$.  Then $X$ consists of periodic points, with precisely $m_i$ distinct orbits of minimal period $n_i$, for $1\leq i\leq s$.  Thus 
\begin{equation*}
\Aut(X)\cong S(n_1,m_1)\times\cdots\times S(n_s,m_s).    \hfill\qedhere
\end{equation*}
\end{proof}

\bigskip
\noindent
{\bf Acknowledgment}:  We thank Jim Davis for pointing us to reference~\cite{SW}.

\end{document}